\documentclass[11pt]{article}
\usepackage{fullpage}
\usepackage{amsmath,amsfonts,amsthm,amssymb,bbm}
\usepackage{mathtools}
\usepackage{xcolor}
\usepackage{mathrsfs}
\usepackage{hyperref}
\usepackage[capitalise]{cleveref}

\theoremstyle{plain}
\newtheorem{theorem}{Theorem}[section]
\newtheorem{lemma}[theorem]{Lemma}
\newtheorem{proposition}[theorem]{Proposition}
\newtheorem{corollary}[theorem]{Corollary}
\theoremstyle{definition}
\newtheorem{definition}[theorem]{Definition}
\theoremstyle{remark}
\newtheorem*{example}{Example}

\newcommand{\gen}[1]{\langle#1\rangle}
\newcommand{\clos}{\mathop{\mathrm{cl}}}
\newcommand{\limfty}[1]{\lim_{#1\rightarrow\infty}}
\newcommand{\ct}{^{\mathrm{T}}}
\newcommand{\sort}{^\downarrow}

\newcommand{\nn}{\mathbb{N}}

\begin{document}
	
	\title{Universality for Doubly Stochastic Matrices}
	\author{Wei Zhan\thanks{Department of Computer Science, Princeton University. Research supported by the Simons Collaboration on Algorithms and Geometry, by a Simons Investigator Award and by the National Science Foundation grant No. CCF-1714779.}}
	
	\date{}
	
	\maketitle
	
	\begin{abstract}
		We show that the set of entries generated by any finite set of doubly stochastic matrices is nowhere dense, in contrast to the cases of stochastic matrices or unitary matrices. In other words, there is no finite universal set of doubly stochastic matrices, even with the weakest notion of universality.
		
		Our proof is based on a theorem for topological semigroups with the convergent property. A topological semigroup is convergent if every infinite product converges. We show that in a compact and convergent semigroup, under some restrictions, the closure of every finitely generated subsemigroup can be instead generated directly by the generating set and the limits of infinite products.                         
	\end{abstract}

	\section{Introduction}
	
	When studying circuit models of computation, a small universal set of gates is often desired, both for the theoretical convenience of capturing their computational power, and for the practical feasibility of physically synthesizing the circuits. For instance, the set $\{\mathrm{AND},\mathrm{OR},\mathrm{NOT}\}$ is universal for classical computation; the set of Hadamard, Toffoli and Phase gates is universal for quantum computation \cite{shi2003both}; and Toffoli gate itself is universal for classical reversible computation \cite{toffoli1980reversible}.
	
	The universality of quantum gates can be translated in a more algebraic language. Let $\mathcal{U}(m)$ be the group of unitary operators on $m$ qubits, and each quantum gate is simply an operator in $\mathcal{U}(m)$. Let $S_m\subset \mathcal{U}(m)$ be the subgroup induced by the permutations on $\{1,\ldots,m\}$, then the collection of circuits built from a set of gates $G$ is exactly the subsemigroup $\gen{G,S_m}$ (generated without inversion). To compute an $n$-qubit unitary, we utilize the remaining $(m-n)$ qubits as ancillas, which starts from state $|0^{m-n}\rangle$ and should also end up as $|0^{m-n}\rangle$. Let $\mathcal{L}_{m,n}\subset \mathcal{U}(m)$ be the subgroup that maps $C^{2^n}\otimes |0^{m-n}\rangle$ to itself, then there is a natural homomorphism $\varphi:\mathcal{L}_{m,n}\rightarrow \mathcal{U}(n)$ as the restriction on the first $n$ qubits. Now we can say that $G$ is universal if for every $n$, there exists $m\geq n$ such that $\varphi(\gen{G,S_m}\cap \mathcal{L}_{m,n})$ is dense in $\mathcal{U}(n)$.
	
	For classical reversible computation, we instead consider the universe $\mathcal{S}(m)$, the group of permutations on $\{0,1\}^m$. To compute an $n$-bit reversible function, the $(m-n)$ bits of ancillas should go from $0^{m-n}$ to $0^{m-n}$, so $\mathcal{L}_{m,n}$ is the subgroup that maps $\{0,1\}^n\cdot 0^{m-n}$ to itself, with a natural homomorphism $\varphi:\mathcal{L}_{m,n}\rightarrow \mathcal{S}(n)$. We can still say a set of gates $G$ is universal if for every $n$, there exists $m\geq n$ such that $\varphi(\gen{G,S_m}\cap \mathcal{L}_{m,n})$ is dense in $\mathcal{S}(n)$, with the discrete topology.
	
	With a little bit more specification, this translation also works for classical computation. The universe becomes $\mathcal{H}(m)$, the semigroup of all functions from $\{0,1\}^m$ to itself. A gate $g$ with fan-in $k$ is then identified with the function
	\[(x_1,\ldots,x_m)\mapsto(g(x_1,\ldots,x_k),x_2,\ldots,x_m).\]
	Again, let $\mathcal{L}_{m,n}$ be the subsemigroup that maps $\{0,1\}^n\cdot 0^{m-n}$ into itself, with the natural homomorphism $\varphi:\mathcal{L}_{m,n}\rightarrow \mathcal{H}(n)$, and with the help of the $\mathrm{FANOUT}$ gate
	\[\mathrm{FANOUT}:(x_1,\ldots,x_m)\mapsto(x_1,x_1,x_3,\ldots,x_m),\]
	we can again say a set of gates $G$ is universal if for every $n$ there exists $m\geq n$ such that $\varphi(\gen{G,S_m,\mathrm{FANOUT}}\cap \mathcal{L}_{m,n})$ is dense in $\mathcal{H}(n)$.
	
	In all of the above examples, the (semi)groups have matrix representations that correspond to each type of computation:
	\begin{enumerate}
		\item Quantum computation $\leftrightarrow$ Unitary matrices
		\item Classical reversible computation $\leftrightarrow$ Permutation matrices
		\item Classical deterministic computation $\leftrightarrow$ Deterministic Transition matrices ($0,1$-matrices that has exactly one entry $1$ in each column)
	\end{enumerate}
	With the computation on $m$ (qu)bits represented by $2^m\times 2^m$ matrices, we can also uniformly define $\mathcal{L}_{m,n}$ as the subsemigroup of matrices $M$ such that $M[i,j]=0$ for all $i\notin\{0,1\}^n\cdot 0^{m-n}$ and $j\in\{0,1\}^n\cdot 0^{m-n}$, while $\varphi$ is simply the operation of taking the submatrix on $\{0,1\}^n\cdot 0^{m-n}$.
	Then it becomes very natural to extend the above discussion on universality to classical randomized computation, as we also have the correspondence:
	\begin{enumerate}
		\setcounter{enumi}{3}
		\item Classical randomized computation $\leftrightarrow$ Stochastic matrices
	\end{enumerate}
	Consider the universe $\mathcal{P}(m)$, the semigroup of all $2^m\times 2^m$ stochastic matrices. The probabilistic part of the computation can be reflected through the single-bit $\mathrm{COINFLIP}$ gate $\begin{pmatrix}
		1/2 & 1/2 \\
		1/2 & 1/2
	\end{pmatrix}$. Then the set of randomized circuits with $m$ bits can be represented by the subsemigroup
	\[\mathcal{R}_m\coloneqq\gen{\mathrm{AND},\mathrm{OR},\mathrm{NOT},\mathrm{COINFLIP},S_m,\mathrm{FANOUT}}.\]
	Since each stochastic matrix is a convex combination of deterministic transition matrices \cite{davis1961markov}, it is not hard to show that indeed for every $n$, there exists $m\geq n$ such that $\varphi(\mathcal{R}_m\cap \mathcal{L}_{m,n})$ is dense in $\mathcal{P}(n)$. In this sense, we can also say that the gate set $\{\mathrm{AND},\mathrm{OR},\mathrm{NOT},\mathrm{COINFLIP}\}$ is universal for classical randomized computation.
	
	Now we turn to the main focus of this paper, doubly stochastic matrices. To some extent they embody the randomized counterpart of reversible computation (although not reversible themselves), as they are exactly the convex combinations of permutation matrices thanks to Birkhoff's Theorem. Computation models based on doubly stochastic matrices have been studied in the regimes of automata \cite{golovkins2002probabilistic,yamasaki2002one}, branching programs \cite{ablayev2005computational,braverman2014pseudorandom,reingold2013pseudorandomness} and distributed computation \cite{nedic2009distributed}.
	
	Considering the fact that both quantum and classical reversible computation possess a finite universal gate set (where the Toffoli gate plays an essential role in both cases), it is tempting to acquire an analogous result for doubly stochastic matrices. However, we have the following fact: Given any finite set of doubly stochastic matrices, if $x$ is the largest entry smaller than $1$, then any product of matrices from the set could only have entries either $1$ or no larger than $x$. That means not only any finite set of doubly stochastic matrices cannot be universal similarly to the previous cases, such a finite set cannot even generate entries that are dense in $[0,1]$.
	
	In fact, the main result of this paper shows something even stronger:
	\begin{theorem}(informally stated) \label{theorem:main}
		The set of entries generated by any finite set of doubly stochastic matrices is nowhere dense.
	\end{theorem}
	Therefore even with the weakest notion of universality, like being able to approximate the values in some positive-length interval, there is still no finite universal set of doubly stochastic matrices. This also partially explains why computation with doubly stochastic matrices did not receive many highlights, despite being a natural analogue and arising from the study of many other computational models.
	
	\subsection{Proof Overview}
	
	To give an intuition of the proof for \cref{theorem:main}, we start from the simple case described below. For any finite set $P\subset\nn$, let $A_P$ be the doubly stochastic transformation that averages over $P$. In matrix representation, it is the direct sum of a matrix over dimensions in $P$ with all entries being $1/|P|$, and the identity over the rest of dimensions.
	
	Consider the finite set of doubly stochastic matrices being $\mathcal{M}=\{A_{\{i,j\}}:1\leq i<j\leq n\}$. If we keep iteratively applying $A_{\{1,2\}},A_{\{2,3\}}$ and $A_{\{1,3\}}$, the overall transformation will approach $A_{\{1,2,3\}}$. Similarly, we can get arbitrarily close to $A_P$ for every $P\subseteq\{1,\ldots,n\}$. Hence the entries in $\gen{\mathcal{M}}$ have limits $1/k$ for each $k\in\{2,\ldots,n\}$. For every $k\in\{2,\ldots,n-2\}$, since we can approach $A_{\{1,\ldots,k\}}$ and $A_{\{k+1,\ldots,n\}}$, we can also approach the matrix
	\[B\coloneqq A_{\{1,k+1\}}A_{\{1,\ldots,k\}}A_{\{k+1,\ldots,n\}},\]
	which provides a limit value $(1/k+1/(n-k))/2$ for the entries in $\gen{\mathcal{M}}$. Then matrix $B$ can be further used as build blocks to approach other matrices, such as $B^2$ or even $A_{\{1,\ldots,n\}}=\limfty{i}B^i$.
	
	Notice that each $A_P$ is the limit of an infinite product of matrices from $\mathcal{M}$. In fact, the reverse is also true: Every infinite product of matrices from $\mathcal{M}$ converges, and the limit can be written as a finite product of some $A_P$. We call this the \textit{convergent} property, and the set $\gen{A_P: P\subseteq\{1,\ldots,n\}}$ is called a \textit{convergence core} for $\mathcal{M}$. Loosely speaking, a topological semigroup is convergent if every infinite product converges, and a convergence core $\mathcal{A}$ of a subset $\mathcal{M}$ satisfies that every infinite product in $\mathcal{M}\cup \mathcal{A}$ converges in $\gen{\mathcal{M}\cup \mathcal{A}}$. Formal definitions of these concepts are given in \cref{section:con}.
	
	In the above we presented some limit points of $\gen{\mathcal{M}}$, which are all constructed using the limits of infinite products from $\mathcal{M}$, and thus are all in $\gen{A_P}$. Yet a priori, a limit point of $\gen{\mathcal{M}}$ could also just be the limit of a sequence in $\gen{\mathcal{M}}$ without any pattern. However, we actually show that the limit points of $\gen{\mathcal{M}}$ are indeed all generated by $\{A_P\}$. More generally, we have the following result:
	
	\begin{theorem}(informally stated)\label{theorem:con}
		In a first-countable, compact and convergent topological semigroup, if a finite subset $\mathcal{M}$ has a relatively small convergence core $\mathcal{A}$, then the closure of $\gen{\mathcal{M}}$ is contained in $\gen{\mathcal{M}\cup \mathcal{A}}$.
	\end{theorem}
	
	The proof of \cref{theorem:con} is presented in \cref{section:omega}. The idea of the proof is that for every sequence in $\gen{\mathcal{M}}$ that converges to some $D\in\clos\gen{\mathcal{M}}$, we can maintain a `left part' $L$ in $\gen{\mathcal{M}\cup \mathcal{A}}$ and a `right sequence' $(R_i)_{i\in\nn}$ in $\gen{\mathcal{M}}$, such that $L\cdot\limfty{i}R_i$ always equals to $D$. The goal is to make $L=D$. Whenever $L\neq D$, we show that it is always possible to extract an infinite product from the left of $(R_i)_{i\in\nn}$, and since it converges in $\gen{\mathcal{M}\cup \mathcal{A}}$ we can append it to the right of $L$. This extract-and-append process might go on indefinitely, but that means $L$ becomes an infinite product itself, and thus can be rewritten as an element in $\gen{\mathcal{M}\cup \mathcal{A}}$, so that the process can continue. Therefore, the time series of this process can be indexed by ordinal numbers. In the actual proof, we construct $L_\alpha$ and $(R_{\alpha,i})_{i\in\nn}$ for all $\alpha<\omega^\omega$ via transfinite induction, and show that there must exist $\alpha$ with $L_\alpha=D$.
	
	The semigroup $\mathcal{D}_n$ of $n\times n$ doubly stochastic matrices is not convergent, i.e. not all infinite products converges, as it contains all permutations on $n$ elements. To handle this, in \cref{section:domestic} we define a subclass of matrices called $\varepsilon$-\textit{domestic matrices} to rule out those matrices that essentially performs permutations. It turns out for sufficiently small $\varepsilon>0$, the set $\mathcal{D}_{n,\varepsilon}$ of $\varepsilon$-domestic matrices is indeed a convergent semigroup, and indeed possess a finite convergence core, which is proved in \cref{section:averaging}. Hence the closure of every finitely generated subsemigroup of $\mathcal{D}_{n,\varepsilon}$ is also finitely generated, and thus countable. Finally, in \cref{section:final} we show that every doubly stochastic matrices can be decomposed into an $\varepsilon$-domestic matrices and a permutation, which combined with above results together proves \cref{theorem:main}.
	
	\subsection{Related Work}
	
	Although the viewpoint of computation as products of matrices has come a long way in quantum computation and (implicitly) in classical computation, the universality has rarely been discussed from this perspective. To the best of our knowledge, the only discussion with similar frameworks of universality was \cite{poritz2013universal}.
	
	Theory of semigroups is a crucial component in algebraic automata theory (see \cite{pin1997syntactic} for an exposition). Results from the study on semigroups of matrices have also found successful application to problems in computer science \cite{kiefer2019efficient}.
	
	Convergence of infinite products of (doubly) stochastic matrices has been extensively studied in the analysis of ergodic Markov chains and consensus algorithms (e.g. \cite{chatterjee1977towards,olfati2004consensus,touri2012backward}). In particular, our result on the convergent property of $\varepsilon$-domestic matrices can be seen as an improvement of the result in \cite{touri2012backward} on doubly stochastic matrices. Convergent properties for more general set of matrices was also studied in \cite{daubechies1992sets}.
	
	\section{Preliminaries}
	
	For a vector $v$ and a matrix $M$, we use $v[i]$ and $M[i,j]$ to denote their entries at $i$-th row and $j$-th column. The symmetric group over $n$ elements is denoted by $S_n$. For a permutation $P\in S_n$, we denote its action on an element $i\in\{1,\ldots,n\}$ and a subset $X\subseteq\{1,\ldots,n\}$ as $Pi$ and $PX$. We also abuse the notation to use $P$ as an $n\times n$ permutation matrix, where $P[i,j]=1$ if and only if $i=Pj$.
	
	A non-negative real square matrix is \textit{stochastic} if every column sums to $1$. A stochastic matrix is \textit{doubly stochastic} if both every row and every column sums to $1$. The set of all $n\times n$ doubly stochastic matrices is denoted as $\mathcal{D}_n$. We use $|\cdot|_1$ and $|\cdot|_\infty$ to denote the $\ell_1$ and $\ell_\infty$ norms of vectors. Let $\triangle_n=\{x\in[0,1]^n:|x|_1=1\}$ be the standard $(n-1)$-simplex. It is well known that $\mathcal{D}_n$ induces a preorder on $\triangle_n$ known as \textit{majorization} (See e.g. \cite{marshall2011inequalities}):
	
	\begin{definition}
		For two vectors $p,q\in\triangle_n$, we say $p$ \textit{majorizes} $q$, denoted as $p\succ q$, if there exists $M\in\mathcal{D}_n$ such that $q=Mp$. Equivalently, $p \succ q$ if and only if for all $k\in\{1,\ldots,n\}$,
		\[\sum_{i=1}^k p\sort[i]\geq \sum_{i=1}^k q\sort[i],\]
		where $p\sort\in\triangle_n$ is vector $p$ sorted in descending order, and similarly is $q\sort$.
	\end{definition}
	
	\subsection*{Ordinal Numbers}
	
	For more formal discussion on ordinal numbers, readers can refer to \cite{jech2013set}. Basically speaking, a set $S$ is an \textit{ordinal number} if it is well-ordered under $\in$, and every element of $S$ is also a subset of $S$. We use lowercase Greek letters $\alpha,\beta,\ldots$ to denote the ordinal numbers.
	
	If $\alpha$ is an ordinal number, then $\alpha+1\coloneqq\alpha\cup\{\alpha\}$ is the successor ordinal of $\alpha$. An ordinal number which is not a successor is called a \textit{limit ordinal}. The smallest nonzero limit ordinal is denoted by $\omega$, which is exactly the set of natural numbers $\mathbb{N}$. For every nonzero ordinal $\alpha$, a transfinite sequence indexed by $\alpha$ is a function $F$ with domain $\alpha$, denoted as $(F(\beta))_{\beta<\alpha}$. When $\alpha=\omega$, this is the infinite sequence in the usual sense, also denoted as $(F_i)_{i\in\nn}$.
	
	Every nonzero ordinal $\alpha$ can be represented in a unique \textit{Cantor Normal Form}:
	\[\alpha=\omega^{\beta_1}\cdot c_1+\cdots+\omega^{\beta_k}\cdot c_k,\]
	where $k,c_1,\ldots,c_k$ are nonzero natural numbers and $\beta_1>\ldots>\beta_k$ are ordinals. We use $\ell(\alpha)$ to denote the lowest exponent $\beta_k$, and $\rho(\alpha)$ to denote the smallest $\gamma$ such that $\gamma+\omega^{\ell(\alpha)}=\alpha$, which can be written as
	\[\rho(\alpha)=\omega^{\beta_1}\cdot c_1+\cdots+\omega^{\beta_{k-1}}\cdot c_{k-1}+\omega^{\beta_k}\cdot(c_k-1).\]
	
	The \textit{transfinite induction} on ordinals goes by the following scheme:
	\begin{theorem}
		Let $P(\alpha)$ be a predicate on ordinals. Assume that:
		\begin{enumerate}
			\item $P(0)$ is true;
			\item $P(\alpha)\Rightarrow P(\alpha+1)$ for every ordinal $\alpha$;
			\item $\left(\bigwedge_{\beta<\alpha} P(\beta)\right)\Rightarrow P(\alpha)$ for every nonzero limit ordinal $\alpha$.
		\end{enumerate}
		Then $P(\alpha)$ is true for every ordinal $\alpha$.
	\end{theorem}
	In \cref{section:omega}, we use transfinite induction on $\omega^\omega$ to prove \cref{theorem:con}. Notice that when $\alpha<\omega^\omega$ we have $\ell(\alpha)\in\mathbb{N}$, so that we can perform subtraction on $\ell(\alpha)$.
	
	\subsection*{Topological Semigroups}
	
	For standard definitions in topology, readers can refer to \cite{willard2012general}. For the sake of simplicity, in this paper we mostly consider first-countable spaces, so that we can talk about closedness and compactness through sequences instead of nets.
	
	A semigroup is a nonempty set equipped with an associative binary operator. A topological semigroup is a semigroup with a Hausdorff topology such that the semigroup operator is continuous. In particular, in a topological semigroup, if both $\limfty{i}A_i$ and $\limfty{i}B_i$ exists, then
	\[\limfty{i}A_iB_i=\left(\limfty{i}A_i\right)\left(\limfty{i}B_i\right).\]
	For example, the set of all $n\times n$ complex matrices forms a topological semigroup, with the semigroup operator being matrix multiplications and the topology induced by any matrix norm. The set $\mathcal{D}_n$ of all $n\times n$ doubly stochastic matrices is a subsemigroup of it.
	
	An element $\epsilon$ of a semigroup $\mathcal{S}$ is an \textit{identity} if $A\epsilon=\epsilon A=A$ for all $A\in\mathcal{S}$. For a subset $\mathcal{M}\subseteq\mathcal{S}$, we use $\gen{\mathcal{M}}$ to denote the subsemigroup generated by $\mathcal{M}$. Notice that if $\mathcal{S}$ has an identity $\epsilon$, then we also assume $\epsilon\in\gen{\mathcal{M}}$.
	
	For a subset $X$ of a topological space, we use $\clos X$ to denote its closure. For sanity check we need the following proposition, so that $\clos\gen{\mathcal{M}}$ is always a subsemigroup in any topological semigroup.
	
	\begin{proposition}\label{lemma:closure}
		In a topological semigroup $\mathcal{S}$, if $\mathcal{T}$ is a subsemigroup, then $\clos \mathcal{T}$ is also a subsemigroup.
	\end{proposition}
	
	Finally, a topological semigroup is \textit{profinite} if it is a projective limit of finite discrete semigroups (See \cite{kyriakoglou2017profinite} for a detailed survey on the subject). Profinite semigroup plays an important role in the study of formal language and automata theory \cite{almeida2005profinite}. In this paper we make use of the following result:
	
	\begin{theorem}[\cite{almeida2005profinite}]\label{theorem:profinite}
		A compact topological semigroup is profinite if and only if it is totally disconnected.
	\end{theorem}
	
	\section{Convergent Topological Semirgoups}\label{section:con}
	
	\begin{definition}
		The \textit{(right) infinite product} $\prod_{i\in\nn} S_i$ in a topological semigroup $\mathcal{S}$ is the limit of the sequence $(S_0\cdots S_i)_{i\in\nn}$. If the limit exists in $\mathcal{S}$, we say the infinite product \textit{converges}.
	\end{definition}
	
	\begin{definition}
		A topological semigroup is \textit{(right) convergent} if every right infinite product converges.
	\end{definition}

	One can define left infinite product and left convergence symmetrically. In the rest of this paper, we stick with the right versions, and omit `right' for succinctness.

	\begin{example}
		Here are some examples of convergent semigroups:
		\begin{itemize}
			\item The semigroup $[0,1]$ with multiplication and the standard topology is convergent.
			\item The extended non-negative real line $[-\infty,+\infty]$ with semigroup operator $\max$ is convergent.
			\item Let $M_1=\begin{pmatrix}1 & 1/2 \\ 0 & 1/2\end{pmatrix}$ and $M_2=\begin{pmatrix}1/2 & 0 \\ 1/2 & 1\end{pmatrix}$. Then the semigroup of matrices $\clos\gen{M_1,M_2}$ with matrix multiplication is convergent. However it is not convergent with respect to left infinite products.
		\end{itemize}	
	\end{example}
	
	One may notice that none of the above examples is topological group. In fact we have the more general results:
	
	\begin{lemma}\label{lemma:monotone}
		In a convergent semigroup $\mathcal{S}$, if $AS=B$ and $BT=A$ then $A=B$.
	\end{lemma}
	\begin{proof}
		The infinite product $ASTST\cdots$ is the limit of the sequence $(A,B,A,B,A,\ldots)$, which exists only if $A=B$ as $\mathcal{S}$ is Hausdorff.
	\end{proof}

	\begin{corollary}\label{corollary:monotone}
		In a convergent semigroup $\mathcal{S}$, if an infinite product $\prod_{i\in\nn} S_i=A$ and $S_0=A$, then for every $i\in\nn$, $S_0\cdots S_i=A$.
	\end{corollary}
	\begin{proof}
		Since $\mathcal{S}$ is convergent, for every $k\in\nn$ the infinite product $\prod_{i>k}S_i$ converges, and  $A=A(S_1\cdots S_k)\cdot\prod_{i>k}S_i$. Let $A(S_1\cdots S_k)$ be $B$ in \cref{lemma:monotone}, then we have $S_0\cdots S_k=A(S_1\cdots S_k)=A$.
	\end{proof}

	\begin{lemma}\label{lemma:absorb}
		In a convergent semigroup $\mathcal{S}$, if an infinite product $\prod_{i\in\nn} S_i=A$ and $S$ appears infinitely many times in $(S_i)_{i\in\nn}$, then $AS=A$.
	\end{lemma}
	\begin{proof}
		Suppose $(i_k)_{k\in\nn}$ is the strictly increasing sequence of indices such that $S_{i_k}=S$. Then
		\[AS=\limfty{k}S_0\cdots S_{i_k-1}\cdot S=\limfty{k}S_0\cdots S_{i_k}=A.\qedhere\]
	\end{proof}
	
	\begin{definition}\label{definition:po}
		In a convergent semigroup $\mathcal{S}$, we define $A\lhd B$ if there exists $S,T\in\mathcal{S}$ such that $B=SAT$, and $A\neq AT$.
	\end{definition}
	
	\begin{proposition}
		For every convergent semigroup, the binary relation $\lhd$ is a strict partial order.
	\end{proposition}
	\begin{proof}
		It suffices to prove that $\lhd$ is transitive and irreflexive.
		\begin{itemize}
			\item (Transitivity) Suppose $A\lhd B$ and $B\lhd C$, where $B=S_1AT_1$ and $C=S_2BT_2$. Then $C=S_2S_1AT_1T_2$, and since $A\neq AT_1$ we have $A\neq AT_1T_2$ by \cref{lemma:monotone}. Thus $A\lhd C$.
			\item (Irreflexivity) Assume by contraction $A\lhd A$, that is $A=SAT$ and $A\neq AT$ for some $A,S,T\in\mathcal{S}$. Let $S^\omega=\prod_{i\in\nn}S$ and $T^\omega=\prod_{i\in\nn}T$. Then we have
			\[A=SAT=S^2AT^2=\ldots=\limfty{i}S^iAT^i=\limfty{i}S^i\cdot A\cdot \limfty{i}T^i=S^\omega AT^\omega.\]
			Therefore $AT^\omega=SAT\cdot \limfty{i}T^i=SAT^\omega$, and
			\[AT^\omega=SAT^\omega=S^2AT^\omega=\ldots=\limfty{i}S^iAT^\omega=S^\omega AT^\omega=A.\]
			That means $AT=A$ by \cref{corollary:monotone}, which is a contradiction.\qedhere
		\end{itemize}
	\end{proof}

	The \textit{height} of a partially ordered set is the maximum cardinality of a chain. Since $\lhd$ is a strict partial order on every convergent semigroup, we can define the height of any subset as its height with the partial order induced by $\lhd$.
	
	Finally, we define the concept of convergence cores:

	\begin{definition}
		Given a convergent semigroup $\mathcal{S}$ and a subset $\mathcal{M}\subseteq\mathcal{S}$ , a subset $\mathcal{A}\subseteq\mathcal{S}$ is a \textit{convergence core} of $\mathcal{M}$ if for every infinite product $\prod_{i\in\nn} S_i$ in $\mathcal{M}\cup\mathcal{A}$, there exists $k\in\nn$ such that $\prod_{i>k} S_i\in\mathcal{A}$.
	\end{definition}

	\begin{example}
		In correspondence with the previous examples of convergent semigroups, we have:
		\begin{itemize}
			\item In the semigroup $[0,1]$ with multiplication, for any $x<1$, $\{0\}\cup(x,1]$ is a convergence core of $[0,1]$, and $\{0\}$ is a convergence core of $[0,x]$.
			\item In the semigroup $[-\infty,+\infty]$ with operator $\max$, for any subset $X\subset [-\infty,+\infty]$, $X$ is a convergence core of itself.
			\item In the semigroup $\clos\gen{M_1,M_2}$ with matrix multiplication, the set $\left\{\begin{pmatrix}
			x & x \\ 1-x & 1-x
			\end{pmatrix}:x\in[0,1]\right\}$ is a convergence core of the whole semigroup.
		\end{itemize}
	\end{example}

	\section{$\omega^\omega$-Indexed Sequence of Limits in Convergent Semirgoups} \label{section:omega}
	
	\begin{lemma}\label{lemma:decompose}
		Given a semigroup $\mathcal{S}$ with identity and a subset $\mathcal{M}\subseteq\mathcal{S}$, for every $L\in\mathcal{S}$ and $R\in\gen{\mathcal{M}}$ that $L\neq LR$, there exists $M\in\mathcal{M}$ and $R'\in\gen{\mathcal{M}}$ such that $LR=LMR'$ and $L\neq LM$.
	\end{lemma}
	\begin{proof}
		Suppose $R=M_1\cdots M_k$ where $M_1,\ldots,M_k\in\gen{\mathcal{M}}$. Since $L\neq LM_1\cdots M_k$, there exists a smallest $i\in\{1,\ldots,k\}$ such that $L\neq LM_1\cdots M_i$. That means
		\[L=LM_1\cdots M_{i-1}\neq LM_1\cdots M_i=LM_i.\]
		Let $M=M_i$ and $R'=M_{i+1}\cdots M_k$, then $LR=LM_i\cdots M_k=LMR'$ and $L\neq LM$.
	\end{proof}
	
	\begin{lemma}\label{lemma:sequence}
		Given a first-countable, compact and convergent semigroup $\mathcal{S}$ with identity, and a finite subset $\mathcal{M}\subset\mathcal{S}$ with convergence core $\mathcal{A}$, for every $D\in\clos\gen{\mathcal{M}}$ there exists an $\omega^\omega$-index sequence of pairs $((L_\alpha,(R_{\alpha,i})_{i\in\nn}))_{\alpha<\omega^\omega}$, where $L_\alpha\in\gen{\mathcal{M}\cup\mathcal{A}}$ and $R_{\alpha,i}\in\gen{\mathcal{M}}$, such that:
		\begin{enumerate}
			\item For all $\alpha<\omega^\omega$, $\limfty{i}R_{\alpha,i}$ exists in $\mathcal{S}$, and $L_\alpha\cdot\limfty{i}R_{\alpha,i}=D$.
			\item For all $0<\alpha<\omega^\omega$, either the following two properties hold:
			\begin{itemize}
				\item[-] There exists $S\in\gen{\mathcal{M}\cup\mathcal{A}}$ such that $L_{\alpha}=L_{\rho(\alpha)}S$, and $L_\alpha\neq L_{\rho(\alpha)}$.
				\item[-] If $\ell(\alpha)>0$, then $L_\alpha=\limfty{i}L_{\rho(\alpha)+\omega^{\ell(\alpha)-1}\cdot i}$.
			\end{itemize}
			Or there exists $\beta<\alpha$ such that $L_\beta=D$.
		\end{enumerate}

	\end{lemma}
	
	\begin{proof}
		The proof is by building up the sequence with transfinite induction on $\alpha<\omega^\omega$:
		\begin{itemize}
			\item For $\alpha=0$, let $L_0$ be the identity, and $(R_{0,i})_{i\in\nn}$ be the sequence in $\gen{\mathcal{M}}$ whose limit is $D$.
			\item For each $\alpha<\omega^\omega$, suppose $(L_\alpha,(R_{\alpha,i})_{i\in\nn})$ is already constructed. If $L_\alpha=D$ then let $L_{\alpha+1}=L_\alpha$ and $R_{\alpha+1,i}=R_{\alpha,i}$ for all $i\in\nn$. Otherwise construct $(L_{\alpha+1},(R_{\alpha+1,i})_{i\in\nn})$ as follows:
			\begin{enumerate}
				\item Since $L_\alpha\neq D$, there are infinitely many $i\in\nn$ such that $L_\alpha\neq L_\alpha R_{\alpha,i}$. For those $i$, by \cref{lemma:decompose} there exists $M_i\in\mathcal{M}$ and $R'_{\alpha,i}\in\gen{\mathcal{M}}$ such that $L_\alpha R_{\alpha,i}=L_\alpha M_iR'_{\alpha,i}$ and $L_\alpha\neq L_\alpha M_i$. Since $\mathcal{M}$ is finite, there exists $M\in\mathcal{M}$ such that the set 
				\[I=\{i\in\nn:L_\alpha\neq L_\alpha R_{\alpha,i}\textrm{ and }M_i=M\}\]
				is infinite. Let $L_{\alpha+1}=L_\alpha M$, so we have $L_{\alpha+1}\neq L_\alpha$.
				
				\item Take $(R_{\alpha+1,i})_{i\in\nn}$ as a converging subsequence of $(R'_{\alpha,i})_{i\in I}$, thanks to the compactness of $\mathcal{S}$. Since $(L_{\alpha+1}R_{\alpha+1,i})_{i\in\nn}$ is a subsequence of $(L_\alpha M_iR'_{\alpha,i})_{i\in I}=(L_\alpha R_{\alpha,i})_{i\in I}$, which is further a subseqeunce of $(L_\alpha R_{\alpha,i})_{i\in\nn}$, we have 
				\[
					L_{\alpha+1}\cdot\limfty{i}R_{\alpha+1,i}
					=\limfty{i}L_{\alpha+1}R_{\alpha+1,i}
					=\limfty{i}L_\alpha R_{\alpha,i}=L_\alpha\cdot\limfty{i} R_{\alpha,i}=D.
				\]
			\end{enumerate}
			
			\item For each nonzero limit oridinal $\alpha<\omega^\omega$, suppose $(L_\beta,(R_{\beta,i})_{i\in\nn})$ is already constructed for every $\beta<\alpha$. If there exists $\beta<\alpha$ that $L_\beta=D$ then let $L_\alpha=L_\beta$ and $R_\alpha,i=R_\beta,i$ for all $i\in\nn$. Otherwise construct $(L_{\alpha},(R_{\alpha,i})_{i\in\nn})$ as follows:
			\begin{enumerate}
				\item For each $j\in\nn$, let $\alpha_j=\rho(\alpha)+\omega^{\ell(\alpha)-1}\cdot j$, so $\rho(\alpha)=\alpha_0$ and $\rho(\alpha_{j+1})=\alpha_j<\alpha$. That means for every $j\in\nn$ there exists $S_j\in\gen{\mathcal{M}\cup\mathcal{A}}$ such that $L_{\alpha_{j+1}}=L_{\alpha_j}S_j$, and $L_{\alpha_{j+1}}\neq L_{\alpha_j}$. Therefore
				\[\limfty{j}L_{\alpha_j}=L_{\rho(\alpha)}\cdot\prod_{j\in\nn}S_j=L_{\rho(\alpha)}S_0\cdots S_{k-1}S_k'A\]
				for some $k\geq 1$, $S_k'\in\gen{\mathcal{M}\cup\mathcal{A}}$ and $A\in\mathcal{A}$. Let $L_\alpha$ be the limit $\limfty{j}L_{\alpha_j}$. Since $L_{\rho(\alpha)}S_0=L_{\alpha_1}\neq L_{\rho(\alpha)}$, by \cref{corollary:monotone} we have $L_\alpha\neq L_{\rho(\alpha)}$.
				\item Let $R$ be a cluster point of the sequence $(\limfty{i}R_{\alpha_j,i})_{j\in\nn}$, whose existence is guaranteed by compactness, and thus
				\[L_\alpha R=\limfty{j}L_{\alpha_j}R
				 =\limfty{j}\left(L_{\alpha_j}\cdot\limfty{i}R_{\alpha_j,i}\right)=D.\]
				Now that $R$ is in $\clos\{R_{\alpha_j,i}:i,j\in\nn\}$, simply take $(R_{\alpha,i})_{i\in\nn}$ as a sequence in $\{R_{\alpha_j,i}:i,j\in\nn\}$ that converges to $R$, so that $L_\alpha\cdot\limfty{i}R_{\alpha,i}=D$ holds. \qedhere
			\end{enumerate}
			
		\end{itemize}

	\end{proof}

	\begin{theorem}
		In a first-countable, compact and convergent semigroup $\mathcal{S}$ with identity, if a finite subset $\mathcal{M}\subset\mathcal{S}$ has a finite height convergence core $\mathcal{A}$, then $\clos\gen{\mathcal{M}}\subseteq\gen{\mathcal{M}\cup \mathcal{A}}$.
	\end{theorem}
	\begin{proof}
		We are going to show that if the convergence core $\mathcal{A}$ has finite height, then for every $D\in \clos\gen{\mathcal{M}}$, in the sequence $((L_\alpha,(R_{\alpha,i})_{i\in\nn}))_{\alpha<\omega^\omega}$ constructed in \cref{lemma:sequence} there exists $\alpha<\omega^\omega$ such that $L_\alpha=D$.
		
		Assume the contrary, that for all $\alpha<\omega^\omega$, $L_\alpha\neq D$. In \cref{lemma:sequence} we showed that for all $0<\alpha<\omega^\omega$ there exists $S\in\gen{\mathcal{M}\cup\mathcal{A}}$ such that $L_{\alpha}=L_{\rho(\alpha)}S$. By the construction in \cref{lemma:sequence}, when $\ell(\alpha)>0$ we can write $S=TA_\alpha$, where $T\in\gen{\mathcal{M}\cup\mathcal{A}}$ and $A_\alpha\in\mathcal{A}$.
		
		Now consider the case when $\ell(\alpha)>1$. With the same definitions of $\alpha_j$ and $S_j$, we can further assume $S_j=T_jA_{\alpha_{j+1}}$ for every $j\in\nn$, since $\ell(\alpha_{j+1})=\ell(\alpha)-1>0$. On the other hand, there exists $k\in\nn$ and some $S_k''\in\gen{\mathcal{M}\cup\mathcal{A}}$ such that
		\[A_\alpha=S_k''\cdot\prod_{j>k}S_j=S_k''T_{k+1}A_{\alpha_{k+2}}\cdot\prod_{j>k+1}S_j.\]
		Moreover, we know $A_{\alpha_{k+2}}\neq A_{\alpha_{k+2}}\cdot\prod_{j>k+1}S_j$, since otherwise it means $\prod_{j>k}S_j=S_{k+1}$, which further implies $S_{k+1}S_{k+2}=S_{k+1}$ by \cref{corollary:monotone}. That leads to $L_{\alpha_{j+2}}=L_{\alpha_{j+1}}$, contradicting to the proven fact in \cref{lemma:sequence}. Therefore with \cref{definition:po} we have $A_{\alpha_{k+2}}\lhd A_\alpha$
		
		The above implies the following: for every $\alpha<\omega^\omega$ such that $\ell(\alpha)>1$, there exists $\beta<\alpha$ with $\ell(\beta)=\ell(\alpha)-1$, such that $A_\beta\lhd A_\alpha$. Thus there exists arbitrarily long chains in $\mathcal{A}$, which contradicts to the fact that $\mathcal{A}$ has finite height.
	\end{proof}
	
	\begin{corollary}\label{corollary:profinite}
		In a first-countable, compact and convergent semigroup $\mathcal{S}$ with identity, if a finite subset $\mathcal{M}\subset\mathcal{S}$ has a finite convergence core $\mathcal{A}$, then $\clos\gen{\mathcal{M}}$ is finitely generated. Furthermore, $\clos\gen{\mathcal{M}}$ is a profinite topological semigroup.
	\end{corollary}
	\begin{proof}
		Let $\mathcal{A}'$ be the subset of $\mathcal{A}$ which consists of $A\in\mathcal{A}$ that can be expressed as an infinite product in $\mathcal{M}$. On one hand, $\mathcal{A}'$ is still a convergence core of $\mathcal{M}$, and it is finite and thus finite height, so $\clos\gen{\mathcal{M}}\subseteq\gen{\mathcal{M}\cup \mathcal{A}'}$. On the other hand, since $\mathcal{A}'\subseteq\clos\gen{\mathcal{M}}$ we have $\gen{\mathcal{M}\cup \mathcal{A}'}\subseteq\clos\gen{\mathcal{M}}$, so we actually have $\clos\gen{\mathcal{M}}=\gen{\mathcal{M}\cup \mathcal{A}'}$. Hence $\clos\gen{\mathcal{M}}$ is finitely generated.
		
		Since $\clos\gen{\mathcal{M}}$ is finitely generated, it is countable. Also it is a closed subspace of the compact Hausdorff space $\mathcal{S}$, so it is compact and Hausdorff itself and thus normal. Therefore $\clos\gen{\mathcal{M}}$ is totally disconnected, and by \cref{theorem:profinite} it is profinite.
	\end{proof}

	\section{Doubly Stochastic Matrices and $\varepsilon$-Domestic Matrices}
	
	\subsection{$\varepsilon$-Domestic Matrices as Convergent Semigroups}\label{section:domestic}

	\begin{definition}
		For any $\varepsilon\in(0,(2n)^{-1}]$, a doubly stochastic matrix $M\in\mathcal{D}_n$ is \textit{$\varepsilon$-domestic}, if there do not exist two subsets $X,Y\subseteq\{1,\ldots,n\}$ such that $X\neq Y$, $|X|=|Y|$, and 
		\[\sum_{i\in X,j\in Y} M[i,j]> (1-\varepsilon)|X|.\]
		The set of all $n\times n$ $\varepsilon$-domestic matrices is denoted as $\mathcal{D}_{n,\varepsilon}$.
	\end{definition}
	\begin{proposition}
		For all $\varepsilon\in(0,(2n)^{-1}]$, $\mathcal{D}_{n,\varepsilon}$ is closed under matrix multiplication, and thus forms a semigroup.
	\end{proposition}
	\begin{proof}
		For every $A,B\in \mathcal{D}_{n,\varepsilon}$ and $X,Y\subseteq\{1,\ldots,n\}$ such that $X\neq Y$ and $|X|=|Y|$, we have
		\begin{align}\label{eq:ip}
		\sum_{i\in X,j\in Y} (AB)[i,j]
		&=\sum_{i\in X,j\in Y}\sum_{k=1}^n A[i,k]B[k,j] \nonumber\\
		&=\sum_{k=1}^n\left(\sum_{i\in X} A[i,k]\right)\left(\sum_{j\in Y} A[k,j]\right)
		\eqqcolon\sum_{k=1}^n A[X,k]B[k,Y]
		\end{align}
		where we use $A[X,k]$ and $B[k,Y]$ to denote $\sum_{i\in X} A[i,k]$ and $\sum_{j\in Y} A[k,j]$ respectively. Notice that we have
		\[\sum_{k=1}^n A[X,k]=\sum_{k=1}^n B[k,Y]=|X|,\quad 0\leq A[X,k],B[k,Y]\leq 1,\ \forall k=1,\ldots,n.\]
		Now suppose $Z\subseteq\{1,\ldots,n\}$ consists of the indices of the largest $|X|$ elements in $(A[X,k])_{k=1,\ldots,n}$. There are two possibilities:
		\begin{itemize}
			\item If $Z=Y$, then \cref{eq:ip} is maximized when $B[k,Y]=1$ for all $k\in Y$, thus
			\[\sum_{i\in X,j\in Y} (AB)[i,j]\leq \sum_{k\in Y}A[X,k]=\sum_{i\in X,k\in Y}A[i,k]\leq(1-\varepsilon)|X|.\]
			\item If $Z\neq Y$, then \cref{eq:ip} is maximized when $B[k,Z]=1$ for all but one $k\in Z$, which is denoted by $k_0$. It corresponds to the smallest $A[X,k]$ for $k\in Z$, and $B[k_0,Z]=1-\varepsilon|X|$ since
			\[\sum_{k\in Z}B[k,Y]=\sum_{k\in Z,j\in Y}B[k,j]\leq(1-\varepsilon)|X|.\]
			Therefore for $k\notin Z$, we have $B[k,Z]\leq \varepsilon|X|\leq B[k_0,Z]$, as $\varepsilon\leq (2n)^{-1}$. Thus $Z$ is also the indices of the largest $|X|$ elements in $(B[k,Y])_{k=1,\ldots,n}$, and \cref{eq:ip} is maximized (regardless of the constraint on $\sum_{k\in Z}A[X,k]$) when $A[X,k]=1$ for all $k\in Z$, so
			\[\sum_{i\in X,j\in Y} (AB)[i,j]\leq \sum_{k\in Z}B[k,Y]\leq(1-\varepsilon)|X|.\qedhere\]
		\end{itemize}
	\end{proof}

	In order to show that $\mathcal{D}_{n,\varepsilon}$ is actually convergent, we need the following lemmas:
	
	\begin{lemma}\label{lemma:domestic}
		For all $p,q\in\triangle_n$ and $M\in\mathcal{D}_{n,\varepsilon}$ such that $q=Mp$, $|p-q|_\infty\leq 2n\varepsilon^{-1}|p\sort-q\sort|_\infty$.
	\end{lemma}
	\begin{proof}		
		Suppose $|p\sort-q\sort|_\infty=\delta$, and let $A,B\in S_n$ be the permutations such that $q=Aq\sort$ and $p=Bp\sort$. Let
		\[K=\{k\in\{1,\ldots,n-1\}:p\sort[k]-p\sort[k+1]>\varepsilon^{-1}\delta\}\cup\{0,n\},\]
		We prove in below that for each $k\in K$, $A\{1,\ldots,k\}=B\{1,\ldots,k\}$.
		
		If $k=0$ or $k=n$, the claim is trivially true. For $k\in K$ that $0<k<n$, assume the contrary that $A$ and $B$ send $\{1,\ldots,k\}$ to different sets $X$ and $Y$. Then
		\begin{equation}\label{eq:mix}
			\sum_{i\in X}q[i]=\sum_{i\in X}\sum_{j=1}^n M[i,j]p[j] 
			\eqqcolon \sum_{j=1}^n M[X,j]p[j] 
		\end{equation}
		where we use $M[X,j]$ to denote $\sum_{i\in X} M[i,j]$. Notice that $Y$ consists of the indices of the largest $k$ elements in $p$, while we have $\sum_{j\in Y} M[X,j]\leq (1-\varepsilon)k$, and $0\leq M[X,j]\leq 1$ since $M$ is $\varepsilon$-domestic. Therefore \cref{eq:mix} is maximized when $M[X,j]=1$ for all but one $j\in Y$, which corresponds to the smallest $p[j]$ for $j\in Y$, and there we have $M[X,j]=1-\varepsilon k$. Thus
		\[
			\sum_{i\in X}q[i]=\sum_{j=1}^n M[X,j]p[j]\leq \sum_{j\in Y}p[j]-\varepsilon k p\sort[k]+\varepsilon k p\sort[k+1]<\sum_{j\in Y}p[j]-k\delta.
		\]
		On the other hand,
		\[\sum_{i\in X}q[i]=\sum_{i=1}^k q\sort[i]\geq \sum_{i=1}^k (p\sort[i]-\delta)=\sum_{j\in Y}p[j]-k\delta,\]
		which leads to a contradiction.
		
		Now for each $i\in\{1,\ldots,n\}$, let $k=A^{-1}i$, and let $k_1,k_2\in K$ be the closest pair in $K$ such that $k_1<k\leq k_2$. The result above implies that $A\{k_1+1,\ldots,k_2\}=B\{k_1+1,\ldots,k_2\}$, so $k'=B^{-1}i$ must also satisfy $k_1<k'\leq k_2$. Therefore
		\[|p[i]-q[i]|=|p\sort[k]-q\sort[k']|\leq|p\sort[k_1+1]-p\sort[k_2]|+|p\sort-q\sort|_\infty\leq n\varepsilon^{-1}\delta+|p\sort-q\sort|_\infty.\]
		As $\varepsilon^{-1}\geq 2n$, this concludes $|p-q|_\infty\leq 2n\varepsilon^{-1}|p\sort-q\sort|_\infty$
	\end{proof}

	\begin{lemma}\label{lemma:intermediate}
		For all $p,q,r\in\triangle_n$ such that $p \succ q \succ r$, $|p\sort-q\sort|_\infty\leq 2n|p\sort-r\sort|_\infty$.
	\end{lemma}
	\begin{proof}
		By the equivalent definition of majorization, we have
		\[\sum_{i=1}^k p\sort[i]\geq \sum_{i=1}^k q\sort[i]\geq \sum_{i=1}^k r\sort[i]\geq \sum_{i=1}^k p\sort[i]-k|p\sort-r\sort|_\infty\]
		for all $k\in\{1,\ldots,n\}$. Therefore for each $k\in\{1,\ldots,n\}$,
		\[|p\sort[k]-q\sort[k]|\leq
		\left|\sum_{i=1}^{k-1} p\sort[i]-\sum_{i=1}^{k-1} q\sort[i]\right|+
		\left|\sum_{i=1}^k p\sort[i]-\sum_{i=1}^k q\sort[i]\right|
		\leq (2k-1)|p\sort-r\sort|_\infty\leq 2n|p\sort-r\sort|_\infty.\qedhere\]
	\end{proof}
	
	\begin{theorem}\label{theorem:domestic}
		For all $\varepsilon\in(0,(2n)^{-1}]$, $\mathcal{D}_{n,\varepsilon}$ is a compact and convergent semigroup.
	\end{theorem}
	\begin{proof}
		The compactness is from the fact that $\mathcal{D}_{n,\varepsilon}$ is an intersection of closed sets (thus closed itself), and is a bounded subset of $\mathbb{R}^{n^2}$.
		
		To show that every infinite product $\prod_{i\in\nn}M_i$ in $\mathcal{D}_{n,\varepsilon}$ converges, it suffices to show that for every $p\in\triangle_n$, the trajectory
		\[(p_i\in\triangle_n:p_0=p,p_{i+1}=M_i\ct p_i)_{i\in\nn}\]
		converges. By the compactness of $\triangle_n$, $(p_i)_{i\in\nn}$ has a converging subsequence, which means that for every $\delta>0$ there exists $i\in\nn$ and infinitely many $k>i$ such that $|p_i-p_k|_\infty\leq\delta$. For each such $k$ and every $j$ such that $i<j<k$, by \cref{lemma:intermediate} we have $|p_i\sort-p_j\sort|_\infty\leq 2n|p_i\sort-p_k\sort|_\infty\leq 2n|p_i-p_k|_\infty\leq 2n\delta$. Then \cref{lemma:domestic} further implies that $|p_i-p_j|\leq 4n^2\varepsilon^{-1}\delta$.
		
		In other words, for every $\delta>0$, there exists $i\in\nn$ such that for every $j,k\geq i$, $|p_j-p_k|\leq 8n^2\varepsilon^{-1}\delta$. Therefore $(p_i)_{i\in\nn}$ is a Cauchy sequence, and thus converges.
		
		Finally, the value of every infinite product $\prod_{i\in\nn}M_i$ in $\mathcal{D}_{n,\varepsilon}$ must also be in $\mathcal{D}_{n,\varepsilon}$, since it is the limit of a sequence in $\mathcal{D}_{n,\varepsilon}$, and $\mathcal{D}_{n,\varepsilon}$ is closed.
	\end{proof}

	\subsection{Averagings as the Convergence Core}\label{section:averaging}
	
	\begin{definition}
		A doubly stochastic matrix is $A\in\mathcal{D}_n$ is the \textit{averaging} over a partition $P_1\sqcup\ldots\sqcup P_s=\{1,\ldots,n\}$, if
		\[A[i,j]=\left\{\begin{array}{ll}
			1/|P_t| & \textrm{if }i,j\in P_t\textrm{ for some }t\in\{1,\ldots,s\} \\
			0 & \textrm{otheriwse.}
		\end{array}\right.\]
		The set of all averagings in $\mathcal{D}_n$ is denoted as $\mathcal{A}_n$. 
	\end{definition}
	
	\begin{proposition}
		For every $\varepsilon\in(0,(2n)^{-1}]$, $\mathcal{A}_n\subset\mathcal{D}_{n,\varepsilon}$.
	\end{proposition}
	\begin{proof}
		For each $A\in\mathcal{A}_n$, suppose $A$ is the averaging over the partition $P_1\sqcup\ldots\sqcup P_s=\{1,\ldots,n\}$. For every $X,Y\subseteq\{1,\ldots,n\}$ such that $X\neq Y$ and $|X|=|Y|$, there must be a $P_t$ such that $X\cap P_t\neq\varnothing$, $P_t\not\subset Y$, and $|X\cap P_t|\geq |Y\cap P_t|$. If $|P_t|=1$ then $\sum_{i\in X,j\in Y}A[i,j]\leq |X|-1$, and if $|P_t|>1$ then
		\[\sum_{i\in X,j\in Y}A[i,j]\leq |X|-|X\cap P_t|\left(1-\frac{|Y\cap P_t|}{|P_t|}\right)\leq |X|-1+\frac{1}{P_t}\leq |X|-\frac{1}{2}.\]
		Whichever happens, we always have $\sum_{i\in X,j\in Y}A[i,j]\leq (1-\varepsilon)|X|$, since $\varepsilon\leq(2|X|)^{-1}$.
	\end{proof}
	
	\begin{lemma}\label{lemma:invariant}
		For every $M\in\mathcal{D}_n$ and $p\in\mathbb{R}^n$, if $Mp=p$ then for every $i,j\in\{1,\ldots,n\}$, either $M[i,j]=0$ or $p[i]=p[j]$.
	\end{lemma}
	\begin{proof}
		The proof is by induction on $n$. The case when $n=1$ is trivial. For $n>1$, let
		\[X=\{i\in\{1,\ldots,n\}:p[i]=p\sort[1]\}.\]
		For each $i\in X$, since $p[i]=\sum_{j=1}^n M[i,j]p[j]$ is a convex combination of $p[1],\ldots,p[n]$, it could only have weight on $p[j]$ for $j\in X$. Thus whenever $i\in X$ and $j\notin X$ we have $M[i,j]=0$. That implies $\sum_{i,j\in X}M[i,j]=|X|$, therefore when $i\notin X$ and $j\in X$ we also have $M[i,j]=0$. And naturally when $i,j\in X$ we have $p[i]=p[j]$.
		
		If $X\neq\{1,\ldots,n\}$, then by removing from $M$ and $p$ the dimensions corresponding to $X$, we get matrix $M'$ and vector $p'$. From the above we know $M'$ is still stochastic, and $M'p'=p'$. By induction hypothesis we have either $M[i,j]=0$ or $p[i]=p[j]$ for all $i,j\notin X$, which completes the induction step.
	\end{proof}
	
	\begin{theorem}\label{theorem:averaging}
		For every $\varepsilon\in(0,(2n)^{-1}]$, and for every finite subset $\mathcal{M}\subset\mathcal{D}_{n,\varepsilon}$, $\mathcal{A}_n$ is a convergence core of $\mathcal{M}$.
	\end{theorem}
	\begin{proof}
		Given an infinite product $\prod_{i\in\nn}M_i$ in $\mathcal{M}$, let $\mathcal{M}'\subseteq\mathcal{M}$ consists of all matrices $M\in\mathcal{M}$ that appears infinitely many times in $(M_i)_{i\in\nn}$. As $\mathcal{M}$ is finite, $\mathcal{M}'$ must be non-empty, and there exists $k\in\nn$ that $\prod_{i>k}M_i$ is an infinite product in $\mathcal{M}'$. Let $A=\prod_{i>k}M_i$, then by \cref{lemma:absorb} we have $AM=A$ for all $M\in\mathcal{M}'$.
		
		Define a binary relation $\sim$ on $\{1,\ldots,n\}$ as follows: $i\sim j$ if and only if there exists $M\in\mathcal{M}'$ such that $M[i,j]>0$. For each pair of $i,j$ such that $i\sim j$, suppose $M\in\mathcal{M}'$ has $M[i,j]>0$, then since $M\ct A\ct p=A\ct p$ for all $p\in\triangle_n$, we have $(A\ct p)[i]=(A\ct p)[j]$ by \cref{lemma:invariant}. In particular, if $i\sim j$ then $A[i,k]=A[j,k]$ for all $k\in\{1,\ldots,n\}$.
		
		Now let $\equiv$ to be the equivalence relation on $\{1,\ldots,n\}$ generated by $\sim$, then whenever $i\equiv j$ we have $A[i,k]=A[j,k]$ for all $k\in\{1,\ldots,n\}$. On the other hand, every $M\in\mathcal{M}'$ is a block diagonal matrix that satisfies $M[i,j]=0$ whenever $i\not\equiv j$, so is their products, and furthermore the limit of infinite products. Therefore whenever $i\not\equiv j$ we have $A[i,j]=0$. There is only one unique matrix $A\in\mathcal{D}_n$ satisfying the above constraints, which is the averging over the partition induced by $\equiv$.
	\end{proof}

	\begin{corollary}\label{corollary:domestic}
		For every $\varepsilon\in(0,(2n)^{-1}]$, and for every finite subset $\mathcal{M}\subset\mathcal{D}_{n,\varepsilon}$, $\clos\gen{\mathcal{M}}$ is finitely generated and profinite.
	\end{corollary}
	\begin{proof}
		This is a direct corollary of \cref{theorem:domestic}, \cref{theorem:averaging} and \cref{corollary:profinite}.
	\end{proof}
	
	\subsection{From $\varepsilon$-Domestic to Doubly Stochastic}\label{section:final}
	
	\begin{lemma}\label{lemma:permute}
		For every $M\in\mathcal{D}_n$, there exist $P\in S_n$ and $M'\in\mathcal{D}_{n,\varepsilon}$ for some $\varepsilon\in(0,(2n)^{-1}]$, such that $M=PM'$.
	\end{lemma}
	\begin{proof}
		Let
		\[C=\left\{(X,Y):X,Y\subseteq\{1,\ldots,n\}, |X|=|Y|, \sum_{i\in X,j\in Y}M[i,j]=|X|\right\}.\]
		It suffices to prove that there exists $P\in S_n$ such that $X=PY$ for all $(X,Y)\in C$, since then we can let
		\[\varepsilon=\min\left( (2n)^{-1},\quad 
		 	1-\max_{\substack{X,Y\subseteq\{1,\ldots,n\},X\neq Y,\\ |X|=|Y|, (X,Y)\notin C}}
			\frac{1}{|X|}\sum_{i\in X,j\in Y}M[i,j]\right)>0\]
		and it is straightforward to check that $P^{-1}M\in\mathcal{D}_{n,\varepsilon}$.
		
		Since $M$ is doubly stochastic, $(X,Y)\in C$ implies that $M[i,j]=0$ whenever $i\in X,j\notin Y$ or $i\notin X,j\in Y$. Therefore if $(X_1,Y_1),(X_2,Y_2)\in C$, we can further deduce that
		\[(X_1\cap X_2, Y_1\cap Y_2),\ (X_1\setminus X_2, Y_1\setminus Y_2),\ (X_2\setminus X_1, Y_2\setminus Y_1)\]
		are all in $C$. Hence $C$ can be generated by a subset $C_0\subseteq C$ through coordinate-wise union, where for every pair of distinct $(X_1,Y_1),(X_2,Y_2)\in C_0$ we have $X_1\cap X_2=Y_1\cap Y_2=\varnothing$. There must exist $P\in S_n$ such that $X=PY$ for all $(X,Y)\in C_0$. Since whenever $X_1=PY_1$ and $X_2=PY_2$ we have $X_1\cup X_2=P(Y_1\cup Y_2)$, such $P$ must also satisfy $X=PY$ for all $(X,Y)\in C$.
	\end{proof}
	
	\begin{theorem}
		For every fixed $p,q\in\triangle_n$, and every finite subset $\mathcal{M}\subset\mathcal{D}_n$, the set $\{q\ct Mp:M\in\gen{\mathcal{M}}\}$ is nowhere dense.
	\end{theorem}
	\begin{proof}
		Without loss of generality assume $S_n\subseteq\mathcal{M}$. Let $e=(1,0,\ldots,0)\in\triangle_n$, and let $A,B\in\mathcal{D}_n$ be the ones such that $p=Ae$ and $q=Be$. We can also assume $A,B\in\mathcal{M}$, so that the target set becomes
		\[\{e\ct B\ct MAe:M\in\gen{\mathcal{M}}\}\subseteq\{M[1,1]:M\in\gen{\mathcal{M}}\}.\]
		For every $a\in[0,1]$, define matrix $M_a\in \mathcal{D}_n$ as
		\[M_a[i,j]=\left\{\begin{array}{ll}
			a & \textrm{if }i=j=1 \\
			\frac{1}{n-1}(1-a) & \textrm{if }i=1,j>1\textrm{ or }i>1,j=1 \\
			\frac{1}{n-1}\left(1-\frac{1}{n-1}(1-a)\right) & \textrm{if }i,j>1.
		\end{array}\right.\]
		If $a=M[1,1]$ for some $M\in\mathcal{D}_n$, it is easy to check that $M_a=A'MA'$, where $A'$ is the averaging over $\{1\}\sqcup\{2,\ldots,n\}$. If we further assume $A'\in\mathcal{M}$, we have
		\[\clos \{M[1,1]:M\in\gen{\mathcal{M}}\}=\{M[1,1]:M\in\clos \gen{\mathcal{M}}\}.\]
		Therefore in order to prove $\{q\ct Mp:M\in\gen{\mathcal{M}}\}$ is nowhere dense, it suffices to prove that $\clos \gen{\mathcal{M}}$ is countable.
		
		By \cref{lemma:permute}, every $M\in\mathcal{M}$ can be decomposed into $PM'$ for some $P\in S_n$, $\varepsilon\in(0,(2n)^{-1}]$ and $M'\in\mathcal{D}_{n,\varepsilon}$. Let $\mathcal{M}'$ be the set of such matrices $M'$, and let $\varepsilon_0$ be the smallest among such $\varepsilon$, so that we have $\mathcal{M}'\in\mathcal{D}_{n,\varepsilon_0}$.
		
		Now given $M_1\cdots M_k\in\gen{\mathcal{M}}$ where $M_1,\ldots,M_k\in\mathcal{M}$, suppose $M_i=P_i M_i'$ for $P_i\in S_n$ and $M_i'\in\mathcal{M}'$. Then we have
		\begin{align}\label{eq:decom}
			&M_1\cdots M_k \nonumber\\
			=&P_1M_1'\cdots P_kM_k' \nonumber\\
			=&(P_1P_2\cdots P_k)(P_k^{-1}\cdots P_2^{-1}M_1'P_2\cdots P_k)\cdots(P_k^{-1} P_{k-1}^{-1} M_{k-2}' P_{k-1} P_k)(P_k^{-1} M_{k-1}' P_k)M_k'.
		\end{align}
		Notice that whenever $P\in S_n$ and $M'\in\mathcal{D}_{n,\varepsilon}$, we have $P^{-1}M'P\in\mathcal{D}_{n,\varepsilon}$. Therefore by letting 
		\[\mathcal{M}''=\{P^{-1}M'P:P\in S_n,M'\in\mathcal{M}'\}\]
		which is a finite subset of $\mathcal{D}_{n,\varepsilon_0}$, \cref{eq:decom} showed that for every $M\in\gen{\mathcal{M}}$ there exists $P\in S_n$ and $M''\in\gen{\mathcal{M}''}$ such that $M=PM''$.
		
		Finally, given $\limfty{i}M_i\in\clos\gen{\mathcal{M}}$ where $M_i\in\gen{\mathcal{M}}$, suppose $M_i=P_i M_i''$ for $P_i\in S_n$ and $M_i''\in\gen{\mathcal{M}''}$. Let $P\in S_n$ be the one that appears infinitely many times in $(P_i)_{i\in\nn}$, and let $(M_{i_j}'')_{j\in\nn}$ be a converging subseqeunce of $(M_i'':P_i=P)_{i\in\nn}$. Then
		\[\limfty{i}M_i=\limfty{j}PM_{i_j}''=P\cdot\limfty{j}M_{i_j}''.\]
		Since $\limfty{j}M_{i_j}''\in\clos \gen{\mathcal{M}''}$, which by \cref{corollary:domestic} is finitely generated and thus countable, we conclude that $\clos \gen{\mathcal{M}}$ is also countable, which completes the proof.
	\end{proof}

	\paragraph{Acknowledgement} We thank Zichang Wang for proposing an inspiring question that leads to the formation of this paper. We also thank Ran Raz and Zhiyuan Li for helpful comments and discussions.
	
	\bibliography{profinite}
	\bibliographystyle{alpha}
	
\end{document}